\documentclass[reqno]{amsart}

\usepackage[pdfencoding=auto]{hyperref}
\hypersetup{hidelinks}

\usepackage{tikz}
\tikzset{> =stealth}

\usepackage[british]{babel}
\usepackage{amsfonts,bbm,dsfont}
\usepackage{amssymb,amsmath}
\usepackage{amsthm}
\usepackage{mathtools}
\usepackage{centernot}
\usepackage[shortlabels]{enumitem}
\usepackage[compress]{cleveref}

\theoremstyle{plain}
\newtheorem{theorem}{Theorem}[section]
\newtheorem{lemma}[theorem]{Lemma}
\newtheorem{proposition}[theorem]{Proposition}
\newtheorem{corollary}[theorem]{Corollary}
\theoremstyle{definition}
\newtheorem{definition}[theorem]{Definition}
\theoremstyle{remark}
\newtheorem{remark}[theorem]{Remark}

\newcommand{\op}{{^\mathrm{op}}}
\newcommand{\comp}{{^\mathsf{c}}}
\newcommand{\cl}{c\ell}
\newcommand{\D}{\mathfrak{D}}
\renewcommand{\L}{\mathcal{L}}
\newcommand{\M}{\mathcal{M}}

\newcommand{\Frm}{\mathrm{Frm}}
\newcommand{\BiFrm}{\mathrm{BiFrm}}
\newcommand{\RBiFrm}{\mathrm{RegBiFrm}}
\newcommand{\StrZdBiFrm}{\mathrm{Str0DBiFrm}}
\newcommand{\Top}{\mathrm{Top}}

\newcommand{\Cat}{\mathrm{Cat}}
\renewcommand{\C}{\mathds{C}}
\renewcommand{\P}{\mathfrak{F}}
\newcommand{\Sk}{\mathrm{Sk}}

\begin{document}

\title[Strictly zero-dimensional biframes]{Strictly zero-dimensional biframes and a characterisation of congruence frames}
\author[G. Manuell]{Graham Manuell}
\address{Department of Mathematics and Applied Mathematics, University of Cape Town, Private Bag Rondebosch, Cape Town 7701, South Africa}
\email{graham@manuell.me}
\thanks{I acknowledge financial assistance from the National Research Foundation of South Africa.}


\begin{abstract}
 Strictly zero-dimensional biframes were introduced by Bana\-schew\-ski and Br\"{u}mmer as a class of strongly zero-dimensional biframes including
 the congruence biframes. We consider the category of strictly zero-dimensional biframes and show it is both complete and cocomplete.
 We characterise the extremal monomorphisms in this category and explore the special position that congruence biframes hold in it.
 Finally, we provide an internal characterisation of congruence biframes, and hence, of congruence frames.
\end{abstract}

\maketitle

\setcounter{section}{-1}
\section{Introduction}

It is a remarkable aspect of frame theory that the lattice of congruences on a frame is itself a frame. It is natural to wonder precisely what form
these congruence frames take. Isbell asks for such a characterisation in \cite{Isbell} and there have subsequently been a number of results which provide
some insight. Joyal and Tierney \cite{galoisTheoryGrothendieck} characterise the inclusion of a frame into its congruence frame by a universal property.
More recently, Frith and Schauerte \cite{asymmetricCong} provide a simple bitopological characterisation of the congruence functor.
We would, however, particularly like an \emph{internal} characterisation of congruence frames. In the compact case, this first was obtained by
Dow and Watson \cite{SkulaSpaces} and independently by Isbell \cite{IsbellDissolute}, but we do not believe a general characterisation has
appeared until now.

Frith observes in \cite{FrithCong} that the congruence frame is arguably better viewed as a biframe. All the previous attempts at characterisation can be
viewed from this perspective. The characterisation of Joyal and Tierney suggests \cref{prop:congruence_free_str0dbifrm}, while that of
Dow and Watson yields part of \cref{lem:compact_str0d_implies_congruential}.

In this paper, we provide a simple internal characterisation of general congruence biframes. This then implies a somewhat more complicated characterisation
of congruence frames. The result follows naturally from consideration of strictly zero-dimensional biframes.

Since their introduction in \cite{StrictlyZeroDimensional}, strictly zero-dimensional biframes have mainly been used as a technical tool to simplify
proofs involving congruence biframes. In this paper we attempt to develop a theory of strictly zero-dimensional biframes as objects of study in their
own right. It is hoped that the study of strictly zero-dimensional biframes might provide insight into other aspects of pointfree topology in future.

\section{Background}

For background on frames, biframes and congruence frames see \cite{PicadoPultr}, \cite{Biframes} and \cite{FrithCong}.
For background on topological functors see \cite{JoyOfCats}.

\subsection{Frames, biframes and congruences}

A frame is a complete lattice satisfying the frame distributivity condition
$x \wedge \bigvee_{\alpha \in I} x_\alpha = \bigvee_{\alpha \in I} (x \wedge x_\alpha)$ for arbitrary families $(x_\alpha)_{\alpha \in I}$.
We denote the smallest element of a frame by $0$ and the largest element by $1$. A frame homomorphism is a function between frames which preserves
finite meets and arbitrary joins. The category of frames is called $\Frm$.

A frame homomorphism $f$ is \emph{dense} if $f(x) = 0 \implies x = 0$ and \emph{codense} if $f(x) = 1 \implies x = 1$.
Every codense homomorphism from a zero-dimensional frame is injective.

A congruence on a frame $L$ is an equivalence relation on $L$ that is also a subframe of $L \times L$.
A closed congruence $\nabla_a$ is a congruence generated by the pair $(0,a)$ and has the explicit form $\{(x,y) \,\mid\, x \vee a = y \vee a\}$.
The quotient $L/\nabla_a$ is isomorphic to the upset ${\uparrow} a \subseteq L$. Similarly, an open congruence $\Delta_a$ is generated by $(a,1)$
and the quotient $L/\Delta_a$ is isomorphic to the downset ${\downarrow} a \subseteq L$.

The lattice of all congruences $\C L$ on a frame $L$ is itself a frame and the assignment $\nabla_L\colon a \mapsto \nabla_a$ is an injective
frame homomorphism. The congruences $\nabla_a$ and $\Delta_a$ are complements of each other and together the closed congruences and open congruences
generate $\C L$.

The map $\nabla_L\colon L \to \C L$ is initial amongst all maps $g\colon L \to M$ for which every element in the image of $g$ is complemented in $M$.
This universal property can be used to define a functor $\C\colon \Frm \to \Frm$ so that the family of maps $(\nabla_L)_L$ becomes a natural
transformation $\nabla\colon \mathbbm{1}_\Frm \to \C$. So if $f\colon L \to M$ and $a \in L$, then $\C f(\nabla_a) = \nabla_{f(a)}$.

A biframe is a triple of frames $\L = (\L_0, \L_1, \L_2)$ where $\L_1$ and $\L_2$ are subframes of $\L_0$ which together generate $\L_0$.
The frame $\L_0$ is called the total part of $\L$, while $\L_1$ and $\L_2$ are called the first and second parts respectively.
A biframe homomorphism $h$ is a frame homomorphism $h_0$ between the total parts which preserves the first and second parts. The restrictions of
a biframe homomorphism $h$ to the first and seconds parts are denoted by $h_1$ and $h_2$ respectively.
The category of biframes is called $\BiFrm$.

A biframe homomorphism is said to be \emph{dense} if its total part is dense.
A biframe homomorphism is a \emph{surjection} or \emph{quotient map} if its first and second parts are surjective. The total parts of biframe quotients
are also surjective and quotients of a biframe are induced uniquely by quotients of its total part. Furthermore, a biframe homomorphism factors
through a biframe quotient if and only if its total part does.

A biframe $\L$ is \emph{strictly zero-dimensional} if every element $a \in \L_1$ is complemented in $\L_0$ with its complement $a\comp$ lying in $\L_2$, and
moreover, these complements generate $\L_2$. (The original definition in \cite{StrictlyZeroDimensional} allowed the roles of $\L_1$ and $\L_2$
above to be interchanged, but as has become standard, we will fix a chirality.) The category of strictly zero-dimensional biframes and biframe
homomorphisms will be called $\StrZdBiFrm$.

A congruence frame $\C L$ has a natural biframe structure $(\C L, \nabla L, \Delta L)$ where
$\nabla L \cong L$ is the subframe of closed congruences and $\Delta L$ is the subframe generated by the open congruences.
A congruence biframe is readily seen to be strictly zero-dimensional. Maps of the form $\C f$ are part preserving and so we may think of $\C$
as a functor from the category frames to the category of strictly zero-dimensional biframes.

\subsection{Closure and quotients}

We will make use of the `closure' map $\cl\colon \C L \to \C L$ given by $\cl = \nabla_L \circ (\nabla_L)_*$.
This map assigns every congruence to the largest closed congruence below it. It is easily seen to be monotone, deflationary and idempotent
and to preserve finite meets. Furthermore, if $C \le D$ in $\C L$, then $\cl(D) \le C$ if and only if the canonical map
$L/C \twoheadrightarrow L/D$ is dense. In particular, $L \twoheadrightarrow L/D$ is dense if and only if $\cl(D) = 0$,
in which case we say $D$ is a dense congruence.

It is an oft-touted advantage of the pointfree approach to topology that every frame $L$ has a smallest dense sublocale.
We write $\D_L$ for the corresponding largest dense congruence.
More generally, for every $a \in L$ there is a largest congruence with closure $\nabla_a$, which we will denote by $\partial_a$.
These are precisely the congruences which induce Boolean quotients of $L$. We might call them Boolean congruences, but to avoid
terminological confusion later, we will call them \emph{clear congruences} as in \cite{clearKappaFrames}.
Clear congruences will play an essential role in our characterisation of congruence biframes in \cref{section:clear_elements}.

We now consider the congruence functor in more detail. Let $f\colon L \to M$ be a frame homomorphism. We can describe the map
$\C f\colon \C L \to \C M$ explicitly. If $C$ is a congruence on $L$ then $\C f(C)$ is the congruence generated by the image of $C$ under $f \times f$.
The right adjoint of $\C f$, which we write as $\C f_*$, also has a simple description: it sends congruences on $M$ to their preimages under $f\times f$.

Note that if we think of elements of $\C L$ as equivalence classes of extremal epimorphisms instead of congruences, then $\C f$ computes
the pushout along $f$. Indeed, $\C f_*$ and $\C f$ are closely related to the covariant and contravariant extremal subobject
functors on $\Frm\op$. Either this perspective or the explicit description of $\C f_*$ quickly gives the following lemma.

\begin{lemma}
 A frame homomorphism $f\colon L \to M$ induces a map $\widetilde{f}\colon L/C \to M/D$ if and only if $\C f (C) \le D$.
\end{lemma}

Another result that will come in useful describes the interaction between the closure map and $\C f$ and can be viewed as a pointfree
analogue of the definition of continuity in terms of preimages and the topological closure operator.
\begin{lemma}\label{lem:closure_and_Cf}
 Let $f\colon L \to M$ and take $C \in \C L$. Then $\C f(\cl(C)) \le \cl(\C f(C))$.
\end{lemma}
\begin{proof}
 We have $\cl(C) \le C$ and so $\C f(\cl(C)) \le \C f(C)$. But $\C f(\cl(C))$ is closed since $\cl(C)$ is. Hence
 $\C f(\cl(C)) \le \cl(\C f(C))$ as required.
\end{proof}

We write $q_C\colon L \twoheadrightarrow L/C$ for the quotient map associated to a congruence $C$. Applying the congruence functor to such a map gives
closed quotient of $\C L$ with kernel $\nabla_C$. So we will often identify $\C q_C$ with the quotient $\C L \twoheadrightarrow \C L / \nabla_C$. Indeed, congruences
on $L/C$ can be identified with congruences on $L$ lying above $C$ by the third isomorphism theorem. We will sometimes write $D/C$ for the image of
$D$ under $\C q_C$.

A particular case of the above gives $L / (C \vee \nabla_a) \cong (L/C)/\nabla_{[a]}$
and thus $\nabla_a \vee C = \nabla_b \vee C$ if and only if $(a,b) \in C$. Consequently, we have the following lemma.
\begin{lemma}\label{lem:closure_of_join_with_closed}
 If $q\colon L \twoheadrightarrow L/C$ is a quotient and $a \in L$, then $\cl(C \vee \nabla_a) = \nabla_{q_*q(a)}$ in $\C L$.
\end{lemma}

\section{First parts and the congruence functor}\label{section:first_parts}

Strictly zero-dimensional biframes are best understood as additional structure placed on their first parts. The first part functor is of lesser interest
for general biframes since it fails to be faithful, but its restriction $\P\colon \StrZdBiFrm \to \Frm$ will play a pivotal role in what follows.

If $\M$ is strictly zero-dimensional then the inclusion $\M_1 \hookrightarrow \M_0$ is epic and so $\P$ is faithful.
Furthermore, important classes of morphisms of strictly zero-dimensional biframes can be characterised by their first parts.

A biframe homomorphism from a strictly zero-dimensional biframe is a surjection if and only if its first part is surjective.
Dense biframe maps in $\StrZdBiFrm$ can also be characterised by their first parts.
\begin{lemma}\label{lem:dense_maps_of_str0d_biframes}
 Let $\L$ be a strictly zero-dimensional biframe. A biframe map $f\colon \L \to \M$ is dense if and only if $f_1$ is injective.
\end{lemma}
\begin{proof}
 Observe first that if $a, b \in \L_1$, then $a \le b$ if and only if $a \wedge b\comp = 0$ in $\L_0$.
 In the same way, $f_1(a) \le f_1(b)$ if and only if $f_0(a \wedge b\comp) = f_1(a) \wedge f_1(b)\comp = 0$.
 So if $f_0$ is dense, $f_1$ reflects order and is thus injective. Conversely, if $f_1$ is injective, then $f_0$ is dense,
 since $\L_0$ is generated under joins by elements of the form $a \wedge b\comp$.
\end{proof}

The first part functor $\P\colon\StrZdBiFrm \to \Frm$ has a left adjoint.
\begin{proposition}\label{prop:congruence_free_str0dbifrm}
 The congruence functor $\C\colon \Frm \to \StrZdBiFrm$ is left adjoint to $\P$.
\end{proposition}
\begin{proof}
 This is simply a rephrasing of the usual universal property of $\nabla_L\colon L \to \C L$.
 Let $L$ be a frame and $\M$ a strictly zero-dimensional biframe. Suppose we have a frame homomorphism $f\colon L \to \M_1$.
 Then every element in the image of $f$ has a complement in $\M_0$ and thus by the universal property of $\C L$ there is a unique
 $\overline{f}\colon \C L \to \M_0$ making the diagram commute.
\begin{center}
   \begin{tikzpicture}[node distance=3.5cm, auto]
    \node (CL) {$\C L$};
    \node (M) [right of=CL] {$\M$};
    \node (L) [below of=CL] {$L$};
    \draw[->] (L) to node {$\nabla_L$} (CL);
    \draw[->, dashed] (CL) to node {$\overline{f}$} (M);
    \draw[->] (L) to node [swap] {$f$} (M);
   \end{tikzpicture}
\end{center}
 The commutativity of the diagram ensures that $\overline{f}$ maps $\nabla L$ into $\M_1$. Then $\overline{f}$ maps $\Delta L$ into $\M_2$ since these
 subframes are both generated by the complements of elements in the first parts. So $\overline{f}$ is a biframe homomorphism and $\C$ is left adjoint
 to $\P$.
\end{proof}

Since the unit of the adjunction is an isomorphism, $\C$ is fully faithful and $\Frm$ embeds as coreflective subcategory of $\StrZdBiFrm$.
We will denote the counit by $\chi\colon \C \M_1 \to \M$ and call this the \emph{congruential coreflection} of $\M$.

We can now characterise the (essential) fibre of a frame $L$ under the functor $\P$. We say a strictly zero-dimensional biframe $\M$ is
\emph{strictly zero-dimensional over} $L$ if $L \cong \P \M$.

\begin{proposition}\label{thm:dense_quotients_of_congruence_frame}
 The strictly zero-dimensional biframes over $L$ are precisely the dense quotients of $\C L$ (up to isomorphism).
\end{proposition}
\begin{proof}
 Let $\M$ be a strictly zero-dimensional biframe over $L$ and let  $\chi\colon \C L \to \M$ be its congruential coreflection.
 The triangle identities imply that $\chi_1$ is an isomorphism. Thus, $\chi$ is a dense quotient by \cref{lem:dense_maps_of_str0d_biframes}.
 
 Conversely, suppose $q\colon \C L \to \M$ is a dense biframe quotient. Then $q_1$ is an isomorphism and so $L \cong \M_1$.
 The complements of elements of $\nabla L$ map to complements of elements of $\M_1$ under $q$ and these generate $\M_2$ since the former
 generate $\Delta L$. Thus, $\M$ is strictly zero-dimensional.
\end{proof}

\begin{corollary}\label{cor:fibres_of_P}
 The fibre category $\P^{-1}(L)$ is a preordered class which is dually equivalent to the frame $\C^2 L / \Delta_{\D_{\C L}}$ ---
 the frame of congruences on $\C L$ that lie below $\D_{\C L}$.
\end{corollary}

\begin{remark}
 \Cref{thm:dense_quotients_of_congruence_frame} was first mentioned (without proof) in \cite{asymmetricCong}. In that paper the authors
 discuss the relationship between $\P\colon \RBiFrm \to \Frm$ and $\C\colon \Frm \to \RBiFrm$ in the context of regular biframes. These are not
 adjoint, but it is shown that $\P$ is faithful and $\C$ is its unique pseudosection. Using similar techniques one can show that $\StrZdBiFrm$ is
 the largest full subcategory of $\RBiFrm$ for which $\P$ restricts to a coreflector.
\end{remark}

A result of \cite{BanaschewskiFrithGilmour} states that every compact congruence frame is the congruence frame of a Noetherian frame (and conversely).
In fact, every compact strictly zero-dimensional biframe is of this form.
\begin{proposition}\label{lem:compact_str0d_implies_congruential}
 Every compact strictly zero-dimensional biframe is the congruence biframe of a Noetherian frame.
\end{proposition}
\begin{proof}
 Let $\L$ be a compact strictly zero-dimensional biframe. Every element of $\L_1$ is complemented and thus compact in $\L_0$.
 Hence $\L_1$ is a Noetherian frame. By \cref{thm:dense_quotients_of_congruence_frame},
 the coreflection $\chi\colon \C \L_1 \to \L$ is a dense quotient. But a dense frame homomorphism from a regular frame to a compact frame
 is injective and so $\chi$ is an isomorphism.
\end{proof}
This gives a characterisation of compact congruence frames as those compact frames admitting a strictly zero-dimensional biframe structure.
A similar characterisation in terms of ordered topological spaces was proved in \cite{SkulaSpaces} and \cite{IsbellDissolute}.
We will generalise our characterisation to the non-compact case in \cref{section:clear_elements}.
For now, we note that \cref{lem:compact_str0d_implies_congruential} also immediately yields the following equivalence of categories.
\begin{corollary}
 The functors $\C$ and $\P$ restrict to an equivalence between the category of Noetherian frames and the category of
 compact strictly zero-dimensional biframes.
\end{corollary}

The above equivalence can alternatively be viewed as a restriction of the equivalence between the category of coherent frames and
the category of compact zero-dimensional biframes, which is itself a restriction of the equivalence described in
\cite{StablyContinuousFrames} between the categories of stably continuous frames and compact regular biframes.

\section{The category of strictly zero-dimensional biframes}\label{section:category_of_str0d_biframes}

The functor $\P\colon \StrZdBiFrm \to \Frm$ has strong lifting properties. In fact, we will show that $\P$ is a \emph{solid} or
\emph{semi-topological} functor (see \cite{tholenSolid}). In particular, this implies $\StrZdBiFrm$ is complete and cocomplete.

\Cref{cor:fibres_of_P} shows that we can define a strictly zero-dimensional biframe by its first part $L$ and a dense
congruence on $\C L$. The map sending a frame to its fibre under $\P$ is not functorial, but the situation improves
if we ignore the density requirement.

Consider the functor $\Frm \xrightarrow{\C^2} \Frm \hookrightarrow \Cat$ where $\C^2 = \C \circ \C$ with $\C$ viewed
as a functor from $\Frm$ to $\Frm$. Applying the Grothendieck construction we obtain a category ${\int}\C^2$ and an opfibration
$F\colon {\int} \C^2 \to \Frm$ which is even topological since the above functor factors through the category of suplattices.

The category ${\int}\C^2$ consists of objects $(L, C)$ where $L$ is a frame and $C \in \C^2 L$ and morphisms $f\colon (L, C) \to (L', C')$
where $f\colon L \to L'$ is a frame homomorphism and $\C^2 f (C) \le C'$. The functor $F$ simply sends $(L,C)$ to $L$.
A morphism $f\colon (L,C) \to (L',C')$ is $F$-final if and only if $C' = \C^2 f (C)$ and $F$-initial if and only if $C = \C^2 f_*(C')$.

Since $\C f\colon \C L \to \C L'$ induces a map from $\C L / C$ to $\C L' / C'$ precisely when $\C^2 f (C) \le C'$,
we may equivalently describe the morphisms in this category as frame homomorphisms $f\colon L \to L'$ such that $\C f$ induces a map from
$\C L / C$ to $\C L' / C'$.

Now $\StrZdBiFrm$ corresponds (up to equivalence) to the full subcategory of ${\int}\C^2$ consisting of the objects $(L, C)$ where $C \le \D_{\C L}$.
Furthermore, restricting $F\colon {\int}\C^2 \to \Frm$ to this subcategory gives the first part functor $\P$.

\begin{lemma}
 The inclusion $\StrZdBiFrm \hookrightarrow {\int}\C^2$ has a left adjoint.
\end{lemma}
\begin{proof}
 Suppose $(L, C)$ is an object in ${\int}\C^2$. Consider the quotient map $\eta_L\colon L \twoheadrightarrow L/\nabla_*(C)$.
 We can view $\C\eta_L$ as the quotient $\C L \twoheadrightarrow \C L/\cl(C)$ and hence $\C^2\eta_L(C) = C/\cl(C)$,
 using the element $C/\cl(C) \in \C(\C L / \cl(C))$ to refer to the corresponding element of $\C^2 (L/\nabla_*(C))$.
 So $\eta_L\colon (L, C) \to (L/\nabla_*(C), C/\cl(C))$ is a final morphism in ${\int}\C^2$.
 
 Now take $(M, D) \in {\int}\C^2$ with $D \le \D_{\C M}$ and consider a morphism $f\colon (L,C) \to (M,D)$.
 We show that there is a unique $\widetilde{f}$ making the following diagram commute.
 \begin{center}
   \begin{tikzpicture}[node distance=3.5cm, auto]
    \node (LoverC) {$(L / \nabla_*(C), C/\cl(C))$};
    \node (M) [right of=CL] {$(M, D)$};
    \node (L) [below of=CL] {$(L, C)$};
    \draw[->>] (L) to node {$\eta_L$} (LoverC);
    \draw[->, dashed] (LoverC) to node {$\widetilde{f}$} (M);
    \draw[->] (L) to node [swap] {$f$} (M);
   \end{tikzpicture}
 \end{center}
 We have $\C^2 f(C) \le D \le \D_{\C M}$ and so $\C^2 f(\cl(C)) \le \cl(\C^2 f(C)) = 0$ by \cref{lem:closure_and_Cf}.
 But $\C^2 f(\cl(C)) = \nabla_{\C f(\nabla_*(C))}$ and hence $\C f(\nabla_*(C)) = 0$.
 Thus, $f$ factors through $\eta_L$ to give a unique frame homomorphism $\widetilde{f}\colon L/\nabla_*(C) \to M$ and this lifts to a
 morphism of ${\int}\C^2$ since $\eta_L$ is final.
 Thus $\StrZdBiFrm \hookrightarrow {\int}\C^2$ has a left adjoint and $\eta_L$ is the unit of the adjunction.
\end{proof}

So $\P$ is the composition of the embedding of a reflective subcategory and a topological functor and is therefore solid.
It follows that $\StrZdBiFrm$ is complete and cocomplete. More explicitly, we have the following propositions.

\begin{proposition}
 Consider a diagram $D\colon \mathbb{I} \to \StrZdBiFrm$ with $DX = \C L_X / C_X$ and
 let $(\alpha_X\colon L_X \to L)_{X \in \mathbb{I}}$ be the colimit of $\P D$ in $\Frm$.
 Then the colimit of $D$ consists of the object $\C L / \bigvee_{X} \C^2 \alpha_X(C_X)$ and the morphisms given by lifting each $\C \alpha_X$
 to a map between the appropriate quotients.
\end{proposition}
\begin{proposition}
 Consider a diagram $D\colon \mathbb{I} \to \StrZdBiFrm$ with $DX = \C L_X / C_X$ and
 let $(\beta_X\colon L \to L_X)_{X \in \mathbb{I}}$ be the limit of $\P D$ in $\Frm$.
 Then the limit of $D$ consists of the object $\C L / \bigwedge_{X} (\C^2 \beta_X)_*(C_X)$ and the morphisms given by lifting each $\C \beta_X$
 to a map between the appropriate quotients.
\end{proposition}

\begin{remark}
 Products and coequalisers of strictly zero-dimensional biframes in $\BiFrm$ are strictly zero-dimensional and so the standard biframe constructions
 give an alternative way to compute products and coequalisers in $\StrZdBiFrm$. Furthermore, while the equaliser in $\BiFrm$ of morphisms between strictly
 zero-dimensional biframes need not be strictly zero-dimensional, the construction is easily modified to give an alternative way to compute equalisers
 in $\StrZdBiFrm$. Another way to compute coproducts would be of particular interest; however, I do not know of one at present.
\end{remark}

\section{Biframes of distinguished congruences}\label{section:str0d_biframes_of_congruences}

We can now characterise the monomorphisms and extremal epimorphisms in $\StrZdBiFrm$.
\begin{lemma}\label{lem:monos_in_str0dbifrm}
 A morphism $f\colon \L \to \M$ in $\StrZdBiFrm$ is monic if and only if it is dense if and only if $\P f$ is injective.
\end{lemma}
\begin{proof}
 The result follows immediately from \cref{lem:dense_maps_of_str0d_biframes} since $\P$ is a faithful right adjoint.
\end{proof}

\begin{lemma}\label{lem:epis_in_str0dbifrm}
 A morphism $f\colon \L \to \M$ in $\StrZdBiFrm$ is an extremal epimorphism if and only if it is a regular epimorphism if and only if
 it is a closed quotient.
\end{lemma}
\begin{proof}
 The regular epimorphisms in $\Frm$ are the surjections and $\Frm$ has (regular epi, mono)-factorisations. Since ${\int}\C^2$ is topological
 over $\Frm$, it too has (regular epi, mono)-factorisations and the regular epimorphisms in ${\int}\C^2$ are the final surjections.
 
 So $\StrZdBiFrm$ is a regular-epireflective subcategory of ${\int}\C^2$. The inclusion functor thus preserves and reflects regular epimorphisms
 and $\StrZdBiFrm$ has (regular epi, mono)-factorisations.
 
 A final surjection in ${\int}\C^2$ has the form $q_A\colon (L,C) \twoheadrightarrow (L/A, \C^2 q_A (C))$.
 Hence the regular (and extremal) epimorphisms in $\StrZdBiFrm$ are surjections of the form
 $\C L / C \twoheadrightarrow (\C  L / \nabla_A) / \C q_{\nabla_A} (C) \cong \C L / (\nabla_A \vee C) \cong (\C L / C) / \C q_C(\nabla_A)$,
 which are precisely the closed surjections.
\end{proof}

So the extremal quotients of a strictly zero-dimensional biframe are in bijection with the elements of its total part. This generalises the relationship
between a frame and its congruence frame.

We can stretch the analogy even further. Let $\M$ be a strictly zero-dimensional biframe and let $\chi\colon \C \M_1 \to \M$ be its
congruential coreflection. Every element $a \in \M_0$ may be associated with a congruence $\chi_*(a)$ on $\M_1$, where $\chi_*$ is
the right adjoint of $\chi_0$. In this way, we may view the elements of any strictly zero-dimensional biframe
as certain congruences on the first part.
It can be instructive to think of $\M$ as the frame $\M_1$ equipped with a frame $\M_0$ of \emph{distinguished congruences}.
Indeed, the following lemma shows that these are precisely the congruences which are induced on the first part by extremal quotients.
So from the point of view of the frame $\M_1$, we can only form quotients by distinguished congruences.

\begin{lemma}\label{lem:closed_quotients_of_str0d_biframe}
 If $\M$ is a strictly zero-dimensional biframe over $L$ and $a \in \M_0$ then $\M/\nabla_a$ is strictly zero-dimensional over $L / \chi_*(a)$.
\end{lemma}
\begin{proof}
 Suppose $\M \cong \C L / D$. Then $\M/\nabla_a \cong \C L / (D \vee \nabla_{\chi_*(a)})$
 and we have that $\cl(D \vee \nabla_{\chi_*(a)}) = \nabla_{\chi_*(a)}$ by \cref{lem:closure_of_join_with_closed}.
 Thus, $\M/\nabla_a$ is a dense quotient of $\C L / \nabla_{\chi_*(a)} \cong \C (L / \chi_*(a))$.
\end{proof}

Congruence biframes are strictly zero-dimensional biframes for which all congruences are distinguished. On the other extreme we have $\C L / \D_{\C L}$ ---
the smallest strictly zero-dimensional biframe over $L$. Here the distinguished congruences are the regular elements of $\C L$, the so-called
\emph{smooth} congruences. Hence we find that smooth congruences are distinguished for any strictly zero-dimensional biframe.

\begin{lemma}
 If $\M$ is a strictly zero-dimensional biframe, then $\chi_*(\M)$ contains every smooth congruence on $\M_1$.
\end{lemma}
\begin{corollary}\label{cor:smooth_elements_fixed_by_chi}
 The right adjoint $\chi_*$ preserves first parts.
\end{corollary}

A more interesting example comes from classical topology.
Let $(X, \tau)$ be a topological space and let $\upsilon$ be the topology generated by the closed sets of $X$. Then letting $\sigma$ denote that join
topology $\tau \vee \upsilon$ we obtain a strictly zero-dimensional biframe $\Sk X = (\sigma,\tau,\upsilon)$ called the
\emph{Skula biframe} of $X$ (see \cite{StrictlyZeroDimensional}). This strictly zero-dimensional biframe only permits us to take \emph{spatially induced}
quotients of the first part, i.e., quotients induced by (Skula-closed) subspaces of $X$.

The assignment $X \mapsto \Sk X$ gives a functor from $\Top\op$ to $\StrZdBiFrm$ which has a left adjoint that sends a strictly zero-dimensional
biframe $\M$ to a space whose points are those of $\M_0$ and whose open sets come from $\M_1$. This adjunction restricts to a dual equivalence between
the category of $T_0$ spaces and the category of strictly zero-dimensional biframes with spatial total part.

This allows us to deal with even non-sober $T_0$ spaces in the pointfree setting.
In fact, if $\M$ is a Skula biframe, then the spatial reflection of $\C\P \M$ corresponds to the \emph{sobrification} of the underlying space.
Further discussion can be found in \cite{congruenceThesis}.

So sobrification appears as the spatial aspect of the congruential coreflection. This suggests that the congruential coreflection itself may be
some kind of pointfree analogue of sobrification. There are various senses in which this is the case and some of these will be the subject of a
later paper.

\section{Clear elements and characterising congruence biframes}\label{section:clear_elements}

Let $\M$ be a strictly zero-dimensional biframe and let $\chi\colon \C \M_1 \to \M$ be its congruential coreflection.
By analogy to congruence frames we call elements of $\M_1$ the \emph{closed elements} of $\M$ and we may define a function
$\cl\colon \M_0 \to \M_0$ so that $\cl(a)$ is the largest closed element below $a$. As before, this map is monotone, deflationary and idempotent
and it preserves finite meets.
It will be useful to have a lemma that generalises the topological formula for closure in a subspace.
\begin{lemma}\label{lem:closure_in_subspace}
 Let $q\colon \L \twoheadrightarrow \M$ be a biframe surjection and let $a \in \M_0$.
 Then $\cl(a) = q(\cl(q_*(a)))$.
\end{lemma}
\begin{proof}
 First note that $q(\cl(q_*(a)))$ is a closed element less than $a$. Now suppose $q(c) \le a$.
 Then $c \le q_*(a)$ and so $c \le \cl(q_*(a))$. Therefore, $q(c) \le q(\cl(q_*(a)))$. Since $q_1$ is surjective, every closed element of $\M$ is of
 the from $q(c)$ for some $c \in \L_1$ and thus $q(\cl(q_*(a)))$ is the largest closed element below $a$.
\end{proof}

The next lemma shows that closure interacts well with the right adjoint of $\chi$.
\begin{lemma}\label{lem:closure_preserved_by_chi_adjoint}
 If $a \in \M_0$ then $\chi_*(\cl(a)) = \cl(\chi_*(a))$.
\end{lemma}
\begin{proof}
 By \cref{lem:closure_in_subspace} we have $\chi_*(\cl(a)) = \chi_*\chi(\cl(\chi_*(a)))$.
 But $\chi_*\chi$ fixes closed elements by \cref{cor:smooth_elements_fixed_by_chi} and so the result follows.
\end{proof}

We may now generalise the notion of a clear congruence to elements of any strictly zero-dimensional biframe.
\begin{definition}
 An element $a$ of a strictly zero-dimensional biframe $\M$ is called \emph{clear} if it is the largest element of $\M_0$
 with closure $\cl(a)$.
\end{definition}
\begin{lemma}\label{lem:clear_element_characterisation}
 Let $a \in \M_0$ and let $c = \cl(a)$. The following are equivalent:
\begin{enumerate}[(1)]
  \item $a$ is clear
  \item $\P(\M / \nabla^{\M_0}_a)$ is Boolean
  \item $\chi_*(a)$ is a clear congruence
  \item $\chi_*(a) = \partial_c$.
 \end{enumerate}
\end{lemma}
\begin{proof}
 If $\P(\M / \nabla^{\M_0}_a)$ is Boolean, then $\chi_*(a)$ is clear by \cref{lem:closed_quotients_of_str0d_biframe}.
 If $\chi_*(a)$ is clear, then $\chi_*(a) = \partial_c$ by \cref{lem:closure_preserved_by_chi_adjoint}.
 
 Now suppose $\chi_*(a) = \partial_c$ and take $b \in \M_0$ with $\cl(b) = c$. By \cref{lem:closure_preserved_by_chi_adjoint} we have
 $\nabla^{\M_1}_c \le \chi_*(b) \le \partial_c$ and so $b = \chi\chi_*(b) \le a$ and $a$ is clear.
 
 Finally, suppose that $a$ is clear and take $x \in \M/\nabla^{\M_0}_a$ such that $\cl(x) = 0$.
 Identifying elements of $\M/\nabla^{\M_0}_a$ with the elements of $\M$ lying above $a$ and using \cref{lem:closure_in_subspace},
 we have that $x \ge a$ and $\cl(x) \vee a = a$. Hence $\cl(x) = \cl(a)$. Since $a$ is clear, it follows that $x \le a$.
 But then $x = 0$ in $\M/\nabla^{\M_0}_a$ and so $0$ is clear in $\M/\nabla^{\M_0}_a$.
 
 We now show that $N = \P(\M/\nabla^{\M_0}_a)$ is Boolean. Let $d$ be a dense element of $N$. This means that $\cl(d\comp) = d^* = 0$.
 But $0$ is clear in $\M/\nabla^{\M_0}_a$ and thus $d\comp = 0$.
 So $d = 1$ and $N$ is Boolean.
\end{proof}
\begin{corollary}
 If $a \in \M_0$ is clear, then every element $b \ge a$ is also clear.
\end{corollary}

Unlike the case of congruence biframes, clear elements of general strictly zero-dimensional biframes might sometimes fail to exist.
That is, there may be no clear element with a given closure.
In fact, the existence of all clear elements characterises the congruential strictly zero-dimensional biframes.

\begin{definition}
 We say an element $a$ of a strictly zero-dimensional biframe is \emph{clarifiable} if there is a clear element with the same closure as $a$.
\end{definition}

\begin{theorem}\label{lem:str0d_biframe_congruential_iff_no_missing_clear_elements}
 A strictly zero-dimensional biframe $\M$ is congruential if and only if all of its (closed) elements are clarifiable.
\end{theorem}
\begin{proof}
 If $\M$ is congruential, then every element of $\M$ is clarifiable.
 
 Conversely, suppose every element of $\M_1$ is clarifiable and take $A \in \C \M_1$ such that $\chi(A) = 1$.
 Then $\nabla_a \le A \le \partial_a$ for some $a \in \M_1$ and so $\chi(\partial_a) = 1$. Let $b \in \M_0$ be the clear element with closure $a$.
 Then $\chi_*(b) = \partial_a$ by \cref{lem:clear_element_characterisation}
 and thus $b = \chi\chi_*(b) = \chi(\partial_a) = 1$. But then $a = \cl(b) = 1$ and so $A = 1$.
 Hence $\chi$ is codense. Since $\C \M_1$ is zero-dimensional, $\chi$ is therefore injective and $\M$ is congruential.
\end{proof}

\begin{remark}
 A Skula biframe $\M$ is sober (in the spatial sense) if and only if every \emph{prime} element of $\M_1$ is clarifiable.
 So this is another sense in which the theory of congruence biframes parallels that of sober spaces.
\end{remark}

Given a frame $M$ without a specified biframe structure we can still use \cref{lem:str0d_biframe_congruential_iff_no_missing_clear_elements} to determine
if it is a congruence frame. The resulting characterisation is somewhat more involved, though perhaps it could lead to simpler characterisations in future.

\begin{corollary}
 A frame $M$ is a congruence frame if and only if it admits an idempotent, deflationary meet-semilattice homomorphism $c\colon M \to M$ such that
\begin{enumerate}[(1)]
  \item Every fixed point of $c$ is complemented
  \item The fixed points of $c$ together with their complements generate $M$
  \item Every fibre of $c$ has a maximum.
 \end{enumerate}
 Furthermore, in this case $M$ is isomorphic to $\C L$, where $L$ is the frame of fixed points of $c$, and
 the subframe inclusion $L \hookrightarrow M$ corresponds to the map $\nabla_L$.
\end{corollary}
\begin{remark}
 For completeness we also mention the easier result that $M$ is a \emph{quotient} of a congruence frame if it satisfies the above
 with condition 3 omitted.
\end{remark}

\bibliographystyle{abbrv}

\end{document}